\documentclass[12pt]{amsart}
\usepackage{amsthm, amssymb}
\usepackage{amsfonts, amscd}
\usepackage{epsfig,multicol}
\usepackage{xypic}
\theoremstyle{plain} \numberwithin{equation}{section}
\newtheorem{thm}{Theorem}[section]
\newtheorem{cor}[thm]{Corollary}
\newtheorem{prop}[thm]{Proposition}
\newtheorem{lem}[thm]{Lemma}

\theoremstyle{definition}
\newtheorem{exam}[thm]{Example}
\theoremstyle{remark}
 \newtheorem{rem}{Remark}

\begin{document}
\title[Examples of quasitoric manifolds as special unitary manifolds]{\large \bf Examples of quasitoric manifolds as special unitary manifolds}
\author[Zhi L\"u and Wei Wang]{Zhi L\"u and Wei Wang}
\thanks{Supported in part by grants from NSFC (No. 11371093, No. 11301335 and 11431009).}
\address{School of Mathematical Sciences, Fudan University, Shanghai,
200433, P.R. China.} \email{zlu@fudan.edu.cn}
\address{College of Information Technology, Shanghai Ocean University, 999 Hucheng Huan Road, 201306, Shanghai, P.R.  China}.
\email{weiwang@amss.ac.cn}


\begin{abstract}
This note shows that for each $n\geq 5$ with only $n\not= 6$, there exists a $2n$-dimensional specially omnioriented quasitoric manifold $M^{2n}$ which represents a nonzero element in $\Omega_*^U$. This provides the counterexamples of Buchstaber--Panov--Ray conjecture.
\end{abstract}

\maketitle


\section{Introduction}\label{int}

Let $\Omega_*^U$ denote the ring formed by the unitary bordism classes of all unitary manifolds, where a  unitary manifold is an  oriented closed smooth manifold whose  tangent bundle admits a stably complex structure. In~\cite{dj}, Davis and Januszkiewicz introduced and studied a class of nicely behaved manifolds $M^{2n}$, the so-called quasitoric manifolds (as  the topological versions of  toric varieties), each of which admits a locally
standard $T^n$-action such that the orbit space of the action is homeomorphic to a simple convex polytope. Buchstaber, Panov and Ray showed in~\cite{bpr1}  that each  quasitoric manifold with an omniorientation always admits a compatible tangential stably complex structure, so omnioriented quasitoric manifolds provide abundant examples of unitary manifolds. In particular, they also showed there that each class of $\Omega_{2n}^U$ contains an omnioriented quasitoric $2n$-manifold as its representative (see also~\cite{br}). In addition,  Buchstaber, Panov and Ray in~\cite{bpr} investigated the property of specially omnioriented quasitoric manifolds, and  proved that if $n<5$, then each $2n$-dimensional specially omnioriented quasitoric manifold represents the zero element in $\Omega_{2n}^U$, where the word ``specially" for a specially omnioriented quasitoric manifold means that  the first Chern class vanishes. Furthermore, they posed the following conjecture.

\vskip .2cm
\noindent {\bf Conjecture ($\star$):}{\em  Let $M^{2n}$ be a specially omnioriented quasitoric manifold. Then $M^{2n}$ represents the zero element in $\Omega_{2n}^U$.}

\vskip .2cm
The purpose of this note is to construct some examples of specially omnioriented quasitoric manifolds that are not bordant to zero in $\Omega_*^U$, which  give the negative answer to the above conjecture in almost all possible dimensional cases. Our main result is stated as follows.
\begin{thm}\label{main}
For each $n\geq 5$ with only $n\not= 6$, there exists a $2n$-dimensional specially omnioriented quasitoric manifold $M^{2n}$ which represents a nonzero element in $\Omega_*^U$.
\end{thm}

Our strategy is related to the unoriented bordism theory. Milnor's work tells us in \cite{m} (see also \cite{s1}) that there is an  epimorphism
\[\begin{CD}
F_*: \Omega^U_*@ >>> \mathfrak{N}^2_*
\end{CD}\]
here $\mathfrak{N}_*$ denotes the ring produced by the unoriented bordism classes of all smooth closed manifolds, which was studied clearly by Ren\'{e} Thom \cite{t}, and $\mathfrak{N}^2_*=\{\alpha^2| \alpha\in \mathfrak{N}_*\}$. This actually implies that there is a covering homomorphism
\[\begin{CD}
H_n: \Omega^U_{2n}@ >>> \mathfrak{N}_{n}
\end{CD}\]
which is induced by $\theta_n\circ F_n$, where $\theta_n:\mathfrak{N}^2_{n}\longrightarrow\mathfrak{N}_n$ is defined by mapping $\alpha^2\longmapsto \alpha$. One sees that $\theta_n$ is well-defined because $\mathfrak{N}$ is a $\mathbb{Z}_2$-polynomial ring with generators in dimensions $d \in {\Bbb Z}_{>0}$ for all
$d \not= 2^s-1$, which is an integral domain, and is the inverse of the Frobenius map $\mathfrak{N}\longrightarrow \mathfrak{N}^2\subset \mathfrak{N}$ given by $\alpha\longrightarrow \alpha^2$. On the other hand, Buchstaber and Ray tell us in~\cite{br} that each class of $\mathfrak{N}_n$ contains an $n$-dimensional small cover as its representative, where a small cover is also introduced by Davis and Januszkiewicz in~\cite{dj}, and it is  the real analogue of a quasitoric manifold. In addition, Davis and Januszkiewicz tell us in \cite{dj} that each quasitoric manifold $M^{2n}$ over a simple convex polytope $P^n$ always admits a natural conjugation involution $\tau$ whose fixed point set $M^\tau$ is just a small cover over $P^n$. In particular,  this conjugation involution   $\tau$ is independent of the choices of omniorientations on $M^{2n}$, and by \cite[Corollaries 6.7--6.8]{dj}, one has that the mod 2 reductions of all Chern numbers of $M^{2n}$ with an omniorientation determine all Stiefel--Whitney numbers of $M^\tau$, and in particular,  $\{M^{2n}\}=\{M^\tau\}^2$ as unoriented bordism classes in $\mathfrak{N}_*$.
Thus, $\tau$ induces a homomorphism $\phi_n^\tau: \Omega^U_{2n}\longrightarrow \mathfrak{N}_{n}$, which exactly agrees with the above homomorphism $H_n: \Omega^U_{2n}\longrightarrow \mathfrak{N}_{n}$.

\vskip .2cm

With the above understood, to obtain the counterexamples of Buchstaber--Panov--Ray conjecture, an approach is to construct
the examples of specially omnioriented quasitoric manifolds whose images under $\phi^\tau$ are nonzero in $\mathfrak{N}_*$. We shall see that Stong manifolds play an important role in our argument.

\vskip .2cm
This note is organized as follows. We shall review the notions and basic properties of quasitoric manifolds and small covers, and state the related result of Buchstaber--Panov--Ray on specially omnioriented quasitoric manifolds in Section~\ref{qs}. We shall review the Stong's work on Stong manifolds and construct some nonbounding orientable Stong manifolds in Section~\ref{stong}. In addition, we also calculate the characteristic matrices of Stong manifolds there. In Section~\ref{m-r} we shall construct required examples of omnioriented quasitoric manifolds as special unitary manifolds and complete the proof of our main result.

\section{Quasitoric manifolds and small covers}\label{qs}
Davis and Januszkiewicz in \cite{dj} introduced and studied two kinds of equivariant manifolds--quasitoric manifolds and small covers, whose geometric and algebraic topology has a strong link to the combinatorics of
polytopes. Following~\cite{dj},  let  $$ G_d^n=\begin{cases}
({\Bbb Z}_2)^n & \text{ if } d=1\\
T^n & \text{ if } d=2
\end{cases}
\ \ \text{and }\ \
R_d
=\begin{cases}
{\Bbb Z}_2 & \text{ if } d=1\\
{\Bbb Z} & \text{ if } d=2.
\end{cases}$$
A $G_d^n$-manifold $\pi_d: M^{dn}\longrightarrow P^n$
$(d=1, 2)$ is a smooth closed $(dn)$-dimensional $G_d^n$-manifold
admitting a locally standard $G_d^n$-action such that its orbit
space is a simple convex $n$-polytope $P^n$. Such a $G_d^n$-manifold
is called a {\em small cover} if $d=1$ and a {\em quasitoric
manifold } if $d=2$.

\vskip .2cm
For a simple convex polytope $P^n$, let $\mathcal{F}(P^n)$ denote
the set of all facets (i.e., $(n-1)$-dimensional faces) of $P^n$. We
know from \cite{dj}
 that each $G_d^n$-manifold $\pi_d: M^{dn} \longrightarrow P^n$ determines a {\em characteristic
function} $\lambda_d$  on $P^n$ $$\lambda_d:
\mathcal{F}(P^n)\longrightarrow R_d^n$$ defined by mapping each
facet in $\mathcal{F}(P^n)$ to nonzero elements of $R_d^n$ such that
$n$ facets meeting at each vertex are mapped to  a basis of $R_d^n$.
 Conversely, the pair $(P^n, \lambda_d)$ can be used to reconstruct $M^{dn}$ as follows:
first $\lambda_d$ gives the following  equivalence relation
$\sim_{\lambda_d}$ on $P^n\times G_d^n$
\begin{equation}\label{equiv}
(x, g)\sim_{\lambda_d} (y, h)\Longleftrightarrow \begin{cases} x=y, g=h & \text{ if } x\in \text{\rm int}(P^n)\\
x=y, g^{-1}h\in G_F &\text{ if } x\in \text{\rm int}F\subset
\partial P^n
\end{cases}\end{equation}
then the quotient space $P^n\times G_d^n/\sim_{\lambda_d}$,  denoted
by $M(P^n, \lambda_d)$,  is  the reconstruction of
$M^{dn}$, where $G_F$ is explained as follows:  for each point $x\in
\partial P^n$, there exists a unique face $F$ of $P^n$ such that $x$
is in its relative interior. If $\dim F=k$, then there are $n-k$
facets, say $F_{i_1}, ..., F_{i_{n-k}}$, such that
$F=F_{i_1}\cap\cdots \cap F_{i_{n-k}}$, and furthermore,
$\lambda_d(F_{i_1}), ..., \lambda_d(F_{i_{n-k}})$ determine a
subgroup of rank $n-k$ in $G_d^n$, denoted by $G_F$. This
reconstruction of $M^{dn}$ tells us that the topology of $\pi_d:
M^{dn}\longrightarrow P^n$ can be determined by $(P^n, \lambda_d)$.

\begin{rem}
If we fix an ordering for all facets in $\mathcal{F}(P)$ (e.g., say $F_1, ..., F_m$) , then the characteristic
function $\lambda_d:
\mathcal{F}(P^n)\longrightarrow R_d^n$ uniquely determines a matrix of size $n\times m$ over $R_d$
$$\Lambda_d=(\lambda_d(F_1), \cdots, \lambda_d(F_m))$$
with $\lambda_d(F_i)$ as columns, which is called the {\em characteristic matrix} of $(P^n, \lambda_d)$ or $M(P^n, \lambda_d)$.
\end{rem}

 We may see from this reconstruction of $G_d^n$-manifolds
that  there is also an essential relation between small covers and
quasitoric manifolds over a simple polytope. In fact, given a quasitoric manifold $M(P^n,
\lambda_2)$ over $P^n$, as shown in \cite[Corollary 1.9]{dj}, there
is a natural conjugation involution on $P^n\times T^n$ defined by $(p,
g)\longmapsto (p, g^{-1})$, which fixes $P^n\times ({\Bbb Z}_2)^n$.
Then this involution descends an involution $\tau$ on $M(P^n,
\lambda_2)$ whose fixed point set is exactly a small cover $M(P^n,
\lambda_1)$ over $P^n$, where $\lambda_1$ is the mod 2 reduction of
$\lambda_2$.

\vskip .2cm

 An {\em omniorientation} of a quasitoric manifold $\pi: M(P^n, \lambda_2)\longrightarrow P^n$ is, by definition in~\cite{bpr1}, just one choice of orientations of $M(P^n, \lambda_2)$ and submanifolds $\pi^{-1}(F), F\in \mathcal{F}(P^n)$. Thus, a quasitoric manifold $\pi: M(P^n, \lambda_2)\longrightarrow P^n$ has $2^{m+1}$ omniorientations, where $m$ is the number of all facets of $P^n$. Clearly, the conjugation involution $\tau$ on $M(P^n,
\lambda_2)$ is independent of the choices of omniorientations of  $M(P^n,
\lambda_2)$. Now let $\mathcal{O}(M(P^n, \lambda_2))$ denote the set of  all $2^{m+1}$  omniorientations. Buchstaber, Panov and Ray showed in \cite{bpr1} (also see \cite{bpr}) that for each omniorientation $\mathfrak{o}\in \mathcal{O}(M(P^n, \lambda_2))$, $M(P^n,
\lambda_2)$ with this omniorientation $\mathfrak{o}$ always admits a tangential stably complex structure, so it is a unitary manifold. In~\cite{bpr}, Buchstaber, Panov and Ray gave a characterization for $M(P^n,
\lambda_2)$ with  $\mathfrak{o}\in \mathcal{O}(M(P^n, \lambda_2))$ to be a special unitary manifold in terms of $\lambda_2$, which is stated as follows.

\begin{prop}[\cite{bpr}]\label{special}
Let $M(P^n, \lambda_2)$ be a quasitoric manifold. Then $M(P^n,
\lambda_2)$ with an omniorientation $\mathfrak{o}\in \mathcal{O}(M(P^n, \lambda_2))$ is a special unitary manifold if and only if there exists a matrix $\sigma$ in $\text{\rm GL}_n({\Bbb Z})$ such that for each facet $F\in \mathcal{F}(P^n)$, the sum of all entries of $\sigma\circ\lambda_2(F)$ is exactly $1$.
\end{prop}


\section{Stong manifolds}\label{stong}

\subsection{Stong manifolds} In~\cite{s}, Stong introduced the Stong manifolds, from which all generators of the unoriented bordism ring $\mathfrak{N}_*$ can be chosen. A Stong manifold is defined as the real projective space bundle
denoted by ${\Bbb R}P(n_1, ..., n_k)$ of the bundle $\gamma_1\oplus\cdots\oplus\gamma_k$ over
${\Bbb R}P^{n_1}\times\cdots\times{\Bbb R}P^{n_k}$, where $\gamma_i$ is the pullback of the canonical bundle over the $i$-th factor ${\Bbb R}P^{n_i}$. The Stong manifold ${\Bbb R}P(n_1, ..., n_k)$ has dimension $n_1+\cdots+n_k+k-1$.
\vskip .2cm

As shown in~\cite{s}, the cohomology with ${\Bbb Z}_2$ coefficients of ${\Bbb R}P(n_1,  ..., n_k)$ is the free module over the cohomology of ${\Bbb R}P^{n_1}\times\cdots\times{\Bbb R}P^{n_k}$ on $1, e, ..., e^{k-1}$, where $e$ is the first Stiefel-Whitney class of the canonical line bundle over ${\Bbb R}P(n_1,  ..., n_k)$, with the relation
$$e^k=w_1e^{k-1}+\cdots+w_re^{k-r}+\cdots+w_k$$
where $w_i$ is the $i$-th Stiefel-Whitney class of $\gamma_1\oplus\cdots\oplus\gamma_k$. 
Then the total Stiefel-Whitney class of ${\Bbb R}P(n_1,  ..., n_k)$ is
\begin{equation} \label{sw class}
\prod_{i=1}^k(1+a_i)^{n_i+1}(1+a_i+e)
\end{equation}
where $a_i$ is the pullback of the nonzero class in $H^1({\Bbb R}P^{n_i};{\Bbb Z}_2)$.
\begin{rem} \label{coh}
In fact, it is easy to see that the total Stiefel-Whitney class of $\gamma_1\oplus\cdots\oplus\gamma_k$ is exactly
$$w(\gamma_1\oplus\cdots\oplus\gamma_k)=\prod_{i=1}^k(1+a_i).$$
So the cohomology with ${\Bbb Z}_2$ coefficients of ${\Bbb R}P(n_1,  ..., n_k)$ may be written as
$${\Bbb Z}_2[a_1, ..., a_k,e]/A$$
where $A$ is the ideal generated by $a_1^{n_1+1}, ..., a_k^{n_k+1}$, and $\prod_{i=1}^k(a_i+e)$.
\end{rem}

Stong further showed in~\cite{s} that

\begin{prop}[\cite{s}]\label{ind}
For $k>1$, ${\Bbb R}P(n_1,  ..., n_k)$ is indecomposable in $\mathfrak{N}_*$ if and only if
$${{m+k-2}\choose {n_1}}+\cdots+{{m+k-2}\choose {n_k}}\equiv 1\mod 2$$
where $m=n_1+\cdots+n_k$.
\end{prop}

Note that generally,  an indecomposable element in $\mathfrak{N}_*$ means that it is not a sum of products of elements of positive degree (see~\cite{s}).

It is not difficult to see from the  expression (\ref{sw class}) of the total Stiefel-Whitney class of ${\Bbb R}P(n_1,  ..., n_k)$  that

\begin{cor}\label{or}
For $k>1$, ${\Bbb R}P(n_1,  ..., n_k)$ is orientable if and only if $k$ and all $n_i$ are even.
\end{cor}

By Proposition~\ref{ind} and Corollary~\ref{or}, we may choose the following examples of indecomposable, orientable Stong manifolds. For $l\geq 0$, ${\Bbb R}P(2, \underbrace{0, ..., 0}_{4l+3})$ and ${\Bbb R}P(4, 2, \underbrace{0,..., 0}_{8l+4})$ are indecomposable and orientable, so they represent nonzero elements in $\mathfrak{N}_*$.
Let $\alpha_{4l+5}$ and $\alpha_{8l+11}$ denote the unoriented bordism classes of ${\Bbb R}P(2, \underbrace{0, ..., 0}_{4l+3})$ and ${\Bbb R}P(4, 2, \underbrace{0,..., 0}_{8l+4})$, respectively. Then we have that

\begin{lem}\label{subring}
All $\alpha_{4l+5}$ and $\alpha_{8l+11}$ with $l\geq 0$ form a polynomial subring $${\Bbb Z}_2[\alpha_{4l+5}, \alpha_{8l+11}|l\geq 0]$$ of $\mathfrak{N}_*$, which contains nonzero classes of dimension $\not=1, 2, 3, 4, 6, 7, 8, 12$.
\end{lem}

\begin{proof}
Because $\alpha_{4l+5}$ and $\alpha_{8l+11}$ are indecomposable in $\mathfrak{N}_*$, any non-trivial polynomial in $\alpha_{4l+5}$ and $\alpha_{8l+11}$ is nonzero in $\mathfrak{N}_*$.
\end{proof}

\subsection{Characteristic matrices of Stong manifolds}
We see that ${\Bbb R}P(n_1, ..., n_k)$ is a ${\Bbb R}P^{k-1}$-bundle over ${\Bbb R}P^{n_1}\times\cdots\times{\Bbb R}P^{n_k}$, so
it is a special generalized real Bott manifold, and in particular, it is also a small cover over $\Delta^{n_1}\times\cdots\times\Delta^{n_k}\times \Delta^{k-1}$,  where $\Delta^l$ denotes an $l$-dimensional simplex.

\begin{rem}
A generalized real Bott manifold is the total space $B^{\Bbb R}_{k+1}$ of an iterated fiber bundle:
\[
\begin{CD} B^{\Bbb R}_{k+1}@ >{\pi_{k+1}}>> B^{\Bbb R}_{k}@>{\pi_{k}}>> \cdots@>{\pi_2}>>B^{\Bbb R}_{1}@
>{\pi_1}>>B^{\Bbb R}_{0}=\{\text{a point}\}
\end{CD}
\]
where each $\pi_i: B^{\Bbb R}_i\longrightarrow B^{\Bbb R}_{i-1}$  is the projectivization of a Whitney sum of $n_i+1$ real line bundles over $B^{\Bbb R}_i$.    It is well-known that  the  generalized real Bott
manifold $B^{\Bbb R}_{k+1}$ is a small cover over $\Delta^{n_1}\times\cdots\times\Delta^{n_{k+1}}$. Conversely, we also know from \cite{cms} that a small cover over a product of simplices is  a generalized real Bott manifold.
\end{rem}

Now let us look at the characteristic matrix of ${\Bbb R}P(n_1, ..., n_k)$ as a small cover over the product $P=\Delta^{n_1}\times\cdots\times\Delta^{n_k}\times \Delta^{k-1}$ with $k>1$ and $n_1\geq n_2\geq\cdots \geq n_k>0$.
Clearly $P$ has $n_1+\cdots +n_k+2k$ facets, which are  listed  as follows:
 $$F_{n_i, j}=\Delta^{n_1}\times\cdots\times\Delta^{n_{i-1}}\times \Delta^{(n_i)}_j\times \Delta^{n_{i+1}}\times\cdots\times\Delta^{n_k}\times \Delta^{k-1}, 1\leq j\leq n_i+1, 1\leq i\leq k $$
 and
 $$F_{k-1, j}=\Delta^{n_1}\times\cdots\times\Delta^{n_k}\times \Delta^{(k-1)}_j, 1\leq j\leq k$$
where $\Delta^{(l)}_j, j=1, ..., l+1$, denote $l+1$ facets of $\Delta^l$.

\vskip .2cm
Throughout the following, we shall carry out our work on  a fixed ordering of all facets of $P=\Delta^{n_1}\times\cdots\times\Delta^{n_k}\times \Delta^{k-1}$ as follows:
$$F_{n_1, 1}, ..., F_{n_1, n_1+1}, ..., F_{n_k, 1}, ..., F_{n_k, n_k+1}, F_{k-1, 1}, ..., F_{k-1, k}.$$

\begin{prop} \label{matrix}
Up to  automorphisms of $({\Bbb Z}_2)^{n_1+\cdots +n_k+k-1}$, the characteristic matrix $\Lambda_1^{(n_1, ..., n_k)}$ of ${\Bbb R}P(n_1, ..., n_k)$ may be written as
\begin{eqnarray*}
\left(
\begin{array}{ccccccccc}
I_{n_1}  &\textbf{1}_{n_1} & &  & &&&\\
&        &   \ddots & &&&& &\\
& &    & I_{n_{k-1}} &\textbf{1}_{n_{k-1}}&&&&\\
& &  & & &I_{n_{k}}&\textbf{1}_{n_{k}}&&\\
&J_{1}  &    \cdots & &J_{k-1}& &\textbf{1}_{k-1} &I_{k-1} & \textbf{1}_{k-1}  \\
\end{array}
\right)
\end{eqnarray*}
with only blocks   $I_{i}$, $\textbf{1}_i$ $(i=n_1, ..., n_k, k-1)$ and $J_j (j=1, ..., k-1)$ being nonzero,  and $0$ otherwise,
where $I_{i}$ denotes the identity matrix of size $i\times i$, $J_j$ denotes the matrix of size $(k-1)\times 1$ with only $(j,1)$-entry being $1$ and $0$ otherwise, and $\textbf{1}_i$ denotes the matrix of size $i\times 1$ with all entries being $1$.
\end{prop}
\begin{proof}
Without the loss of generality, assume that the values of the characteristic function $\lambda_1^{(n_1, ..., n_k)}$ on the following $n_1+\cdots+n_k+k-1$ facets $$F_{n_1, 1}, ..., F_{n_1, n_1}, ..., F_{n_k, 1}, ..., F_{n_k, n_k}, F_{k-1, 1}, ..., F_{k-1, k-1}$$ meeting at a vertex
are all columns with an ordering from the first column to the last column in $I_{n_1+\cdots+n_k+k-1}$, respectively.  It suffices to determine the values of $\lambda_1^{(n_1, ..., n_k)}$
on the  $k+1$ facets
$F_{n_1, n_1+1}, F_{n_2, n_2+1}, ..., F_{n_k, n_k+1},  F_{k-1, k}$.
By \cite[Lemma 6.2]{lt}, we have that for $1\leq i\leq k$
$$\lambda_1^{(n_1, ..., n_k)}(F_{n_i, n_i+1})=\sum_{j=1}^{n_i}\lambda_1^{(n_1, ..., n_k)}(F_{n_i, j})+\beta_i$$
and
$$\lambda_1^{(n_1, ..., n_k)}(F_{k-1, k})=\sum_{j=1}^{k-1}\lambda_1^{(n_1, ..., n_k)}(F_{k-1, j})+\beta_{k+1}$$
such that those entries from $(n_1+\cdots+n_{i-1}+1)$-th to $(n_1+\cdots +n_i)$-th of $\beta_i$ are all zero, and those entries from $(n_1+\cdots+n_k+1)$-th to $(n_1+\cdots +n_k+k-1)$-th of $\beta_{k+1}$ are all zero.
In particular, we also know by \cite[Lemma 6.3]{lt} that there is at least one $\beta_i$ such that $\beta_i=0$ in $({\Bbb Z}_2)^{n_1+\cdots +n_k+k-1}$.

\vskip .2cm Now by \cite[Theorem 4.14]{dj}, we may write $H^*({\Bbb R}P(n_1, ..., n_k);{\Bbb Z}_2)$ as
$${\Bbb Z}_2[F_{n_1, 1}, ..., F_{n_1, n_1+1}, ..., F_{n_k, 1}, ..., F_{n_k, n_k+1}, F_{k-1, 1}, ..., F_{k-1, k}]/I_P+J_{\lambda_1^{(n_1, ..., n_k)}}$$
where the $F_{i, j}$ are used as indeterminants of degree 1, $I_P$ is the Stanley-Reisner ideal generated by $\prod_{j=1}^{n_i+1}F_{n_i, j} (i=1, ..., k)$ and $\prod_{i=1}^kF_{k-1, i}$, and $J_{\lambda_1^{(n_1, ..., n_k)}}$ is the ideal determined by $\lambda_1^{(n_1, ..., n_k)}$. Furthermore, we have by \cite[Corollary 6.8]{dj} that the total Stiefel-Whitney class of ${\Bbb R}P(n_1,  ..., n_k)$ is
$$\prod_{i=1}^k\Big(\prod_{j=1}^{n_i+1}(1+F_{n_i, j})\Big)(1+F_{k-1, i}).$$
Comparing with the formula (\ref{sw class}) or by Remark~\ref{coh}, we see that for each $1\leq i\leq k$,
\begin{equation}\label{equ} F_{n_i, 1}=\cdots= F_{n_i, n_i+1} \text{ (denoted by } a_i) \end{equation}
so $a_i^{n_i+1}=\prod_{j=1}^{n_i+1}F_{n_i, j}=0$. Then we obtain from all equations in (\ref{equ}) that the characteristic matrix $\Lambda_1^{(n_1, ..., n_k)}$ corresponding to $\lambda_1^{(n_1, ..., n_k)}$ is of the form
\begin{eqnarray*}
\left(
\begin{array}{ccccccccc}
I_{n_1}  &\textbf{1}_{n_1} & &  & &&&\\
&        &   \ddots & &&&& &\\
& &    & I_{n_{k-1}} &\textbf{1}_{n_{k-1}}&&&&\\
& &  & & &I_{n_{k}}&\textbf{1}_{n_{k}}&&\\
&B_1  &    \cdots & &B_{k-1}& &B_k &I_{k-1} & \textbf{1}_{k-1}  \\
\end{array}
\right)
\end{eqnarray*}
where all blocks  except for $I_{i}$, $\textbf{1}_i$ $(i=n_1, ..., n_k, k-1)$ and $B_j (j=1, ..., k)$ are  zero.
 This implies that $\beta_{k+1}$ must be the zero element, and for $1\leq i\leq k$, each $\beta_i$ is of the form
 $$(\underbrace{0,...,0}_{n_1+\cdots+n_k}, \beta_{i, 1}, ..., \beta_{i, k-1})^\top$$ in  $({\Bbb Z}_2)^{n_1+\cdots +n_k+k-1}$. Moreover, one has that
 \begin{equation}\label{rel}
 \begin{cases}
 F_{k-1, 1}=F_{k-1, k}+\beta_{1,1}F_{n_1, n_1+1}+\cdots+\beta_{k, 1}F_{n_k, n_k+1}\\
 \cdots\\
 F_{k-1, k-1}=F_{k-1, k}+\beta_{1,k-1}F_{n_1, n_1+1}+\cdots+\beta_{k, k-1}F_{n_k, n_k+1}
 \end{cases}
 \end{equation}
Comparing with the formula (\ref{sw class}) again, one should have that
$$\prod_{i=1}^k(1+F_{k-1, i})=\prod_{i=1}^k(1+a_i+e)=\prod_{i=1}^k(1+F_{n_i, n_i+1}+e).$$
Without the loss of generality, assume that $1+F_{k-1, i}=1+F_{n_i, n_i+1}+e$ for $1\leq i\leq k$. Then for $i=k$, one has that
$e=F_{k-1, k}+F_{n_k, n_k+1}$,
and for $1\leq i<k$, one has by (\ref{rel}) that
$$\beta_{1,i}F_{n_1, n_1+1}+\cdots+\beta_{k, i}F_{n_k, n_k+1}=F_{n_i, n_i+1}+F_{n_k, n_k+1}$$
so $\beta_{i, i}=\beta_{k, i}=1$ and $\beta_{j, i}=0$ if $j\not=i, k$ since $F_{n_1, n_1+1}, ..., F_{n_k, n_k+1}$ are linearly independent in  $H^1({\Bbb R}P(n_1, ..., n_k);{\Bbb Z}_2)$. This completes the proof.
\end{proof}

If there is a minimal integer $i$ with $1\leq i< k$ such that $n_i>0$ but $n_{i+1}=0$ (so $n_j=0$ for $j\geq i+1$), then a similar argument as above gives
\begin{prop} \label{mp}
Suppose that there is some $i$ with  $1\leq i< k$ such that $n_1\geq \cdots \geq n_{i}>0$ and $n_{i+1}=\cdots=n_k=0$.
Up to  automorphisms of $({\Bbb Z}_2)^{n_1+\cdots +n_{i}+k-1}$, the characteristic matrix $\Lambda_1^{(n_1, ..., n_i, 0, ..., 0)}$ of ${\Bbb R}P(n_1, ..., n_i, 0, ..., 0)$ may be written as
\begin{eqnarray*}
\left(
\begin{array}{ccccccc}
I_{n_1}  &\textbf{1}_{n_1} &    &&&\\
&        &   \ddots & &&& \\
& &    & I_{n_{i}} &\textbf{1}_{n_{i}}&&\\
&J_{1}  &    \cdots & &J_{i}  &I_{k-1} & \textbf{1}_{k-1}  \\
\end{array}
\right)
\end{eqnarray*}
with only blocks   $I_{j}$, $\textbf{1}_j$ $(j=n_1, ..., n_i, k-1)$ and $J_l (l=1, ..., i)$ being nonzero,  and $0$ otherwise,
where $I_{j}$,  $J_l$  and $\textbf{1}_j$ represent the same meanings as stated in Proposition~\ref{matrix}.
\end{prop}

\section{Proof of Main Result}\label{m-r}

\subsection{Examples of specially omnioriented quasitoric manifolds} Throughout the following,  for a $k$-dimensional simplex $\Delta^k$,  $\Delta^{(k)}_i, i=1, ..., k+1$ mean the $k+1$ facets of $\Delta^k$, and for a product $P=\Delta^{k_1}\times\cdots\times\Delta^{k_r}$ of simplices,
$F_{k_i, j}$ means that the facet $\Delta^{k_1}\times\cdots\times\Delta^{k_{i-1}}\times \Delta^{(k_i)}_j\times\Delta^{k_{i+1}}\times\cdots\times\Delta^{k_r}$ of $P$. Then let us construct some required examples.

\begin{exam} Let $P^{4l+5}=\Delta^2\times \Delta^{4l+3}$ with $l\geq 0$. Define a characteristic function $\lambda_2^{(2, 0, ..., 0)}$ on $P^{4l+5}$ in the following way.  First let us fix an ordering of all facets of $P^{4l+5}$ as follows
$$F_{2, 1}, F_{2, 2}, F_{2, 3}, F_{4l+3, 1}, ..., F_{4l+3, 4l+3}, F_{4l+3, 4l+4}.$$
 Then we construct the characteristic matrix $\Lambda_2^{(2, 0, ..., 0)}$ of the required characteristic function $\lambda_2^{(2, 0, ..., 0)}$ on the above ordered facets as follows:
\begin{eqnarray*}
\Lambda_2^{(2, 0, ..., 0)}=
\left(
\begin{array}{cccc}
I_2  &\widetilde{\textbf{1}}_2 &    &\\
&J_1  &      I_{4l+3} & \widetilde{\textbf{1}}_{4l+3}  \\
\end{array}
\right)
\end{eqnarray*}
with only blocks   $I_{j}$, $\widetilde{\textbf{1}}_j$ $(j=2, 4l+3)$ and $J_1$ being nonzero,  and $0$ otherwise,
where $I_{j}$ and $J_1$ denote the same meanings as in Proposition~\ref{matrix}, and $\widetilde{\textbf{1}}_j$ denotes the matrix of size $j\times 1$ with $(i, 1)$-entries for all even $i$ being $-1$ and other entries being $1$.
We see that the sum of all entries of each column in the characteristic matrix $\Lambda_2^{(2, 0, ..., 0)}$ is always 1. Thus, by Proposition~\ref{special}, one has that the quasitoric manifold
 $M(P^{4l+5}, \lambda_2^{(2, 0, ..., 0)})$ with the given omniorientation is a special unitary manifold.
\end{exam}

\begin{exam}
Let $P^{8l+11}=\Delta^4\times \Delta^2\times \Delta^{8l+5}$ with $l\geq 0$. In a similar way as above, fix an ordering of all facets of $P^{8l+11}$ as follows:
$$F_{4, 1}, F_{4,2}, F_{4, 3}, F_{4,4}, F_{4,5}, F_{2, 1}, F_{2, 2}, F_{2, 3}, F_{8l+5, 1}, ..., F_{8l+5, 8l+5}, F_{8l+5, 8l+6}.$$
Then we define a characteristic function $\lambda_2^{(4,2, 0, ..., 0)}$ on the above ordered facets of $P^{8l+11}$ by the following characteristic matrix
\begin{eqnarray*}
\Lambda_2^{(4,2, 0, ..., 0)}=
\left(
\begin{array}{ccccccc}
I_4  &\widetilde{\textbf{1}}_4 &   & &&\\
& &     I_2 &\widetilde{\textbf{1}}_2&&&\\
&J_{1}   & &J_2&  &I_{8l+5} & \widetilde{\textbf{1}}_{8l+5}  \\
\end{array}
\right)
\end{eqnarray*}
with only blocks   $I_{i}$, $\widetilde{\textbf{1}}_i$ $(i=2, 4, 8l+5)$ and $J_j (j=1,2)$ being nonzero,  and $0$ otherwise,
where $I_{i}$, $J_j$  and $\widetilde{\textbf{1}}_i$ denote  the  same meanings as above.  By Proposition~\ref{special},
 $M(P^{8l+11}, \lambda_2^{(4,2, 0, ..., 0)})$ with the given omniorientation is a special unitary manifold.
\end{exam}

\begin{exam} \label{7}
The case in which $n=7$. Consider the  polytope $P^7=\Delta^4\times \Delta^3$ with the following ordered  facets
$$F_{4, 1}, F_{4, 2}, F_{4, 3}, F_{4, 4}, F_{4, 5}, F_{3, 1}, F_{3, 2}, F_{3, 3}, F_{3, 4}.$$
Then we may define a characteristic function $\lambda_2^{<7>}$ on the  ordered facets of $P^7$ by the following characteristic matrix
\begin{eqnarray*}
\left(
\begin{array}{ccccccccc}
1 &&&&1&&&&\\
&1 & & &-1 &&&&    \\
 && 1& & 1&&&& \\
 &&&1& -1 &&&&\\
 &&&&1& 1 & & & 1\\
 &&&&&&1&&-1\\
 &&&&&&&1& 1\\
\end{array}
\right),
\end{eqnarray*}
which gives a special unitary manifold $M(P^7, \lambda_2^{<7>})$. Moreover, by the Davis--Januszkiewicz theory, we may read off the cohomology of $M(P^7, \lambda_2^{<7>})$ as follows:
$$H^*(M(P^7, \lambda_2^{<7>}))={\Bbb Z}[x,y]/<x^5, y^4+xy^3>$$
with $\deg x=\deg y=2$, and by~\cite[Theorem 4.8]{dj} and~\cite{bpr}, the total Chern class of $M(P^7, \lambda_2^{<7>})$ may be written as
$$c(M(P^7, \lambda_2^{<7>}))=(1-x^2)^2(1+x)(1-x-y)(1-y^2)(1+y).$$
A direct calculation gives the Chern number $\langle c_3c_4, [M(P^7, \lambda_2^{<7>})]\rangle=-2\not=0$, which implies that this specially omnioriented quasitoric manifold $M(P^7, \lambda_2^{<7>})$
is not bordant to zero in $\Omega_*^U$.
\end{exam}

\begin{exam}\label{8}
The case in which $n=8$. Consider the  polytope $P^8=\Delta^3\times \Delta^5$ with the ordered  facets as follows:
$$F_{3, 1}, F_{3, 2}, F_{3, 3}, F_{3, 4}, F_{5, 1}, F_{5, 2}, F_{5, 3}, F_{5, 4}, F_{5, 5}, F_{5, 6}.$$
Then we may define a characteristic function $\lambda_2^{<8>}$ on the  ordered facets of $P^8$ by
\begin{eqnarray*}
\left(
\begin{array}{cccccccccc}
1 &&&1&&&&&&\\
&1 & & -1&& &&&&    \\
 && 1&  1&&&&&& \\
 &&&-1& 1 &&&&&1\\
 &&&1&  &1 & &&& -1\\
 &&&&&&1&&&1\\
 &&&&&&&1&& -1\\
 &&&&&&&&1 & 1\\
\end{array}
\right),
\end{eqnarray*}
which also gives a special unitary manifold $M(P^8, \lambda_2^{<8>})$. Similarly,  one has the cohomology of $M(P^8, \lambda_2^{<8>})$
$$H^*(M(P^8, \lambda_2^{<8>}))={\Bbb Z}[x,y]/<x^4, y^4(x-y)^2>$$
with $\deg x=\deg y=2$, and the total Chern class of $M(P^8, \lambda_2^{<8>})$
$$c(M(P^8, \lambda_2^{<8>}))=(1-x^2)^2(1-y^2)^2[1-(x-y)^2].$$
Furthermore, one has  the Chern number $\langle c_4^2, [M(P^8, \lambda_2^{<8>})]\rangle=4\not=0$. So  $M(P^8, \lambda_2^{<8>})$
is not bordant to zero in $\Omega_*^U$.
\end{exam}

\begin{exam}\label{12}
The case in which $n=12$. Consider the  polytope $P^{12}=\Delta^3\times \Delta^9$ with the ordered  facets as follows:
$$F_{3, 1}, F_{3, 2}, F_{3, 3}, F_{3, 4}, F_{9, 1}, F_{9, 2}, F_{9, 3}, F_{9, 4}, F_{9, 5}, F_{9, 6}, F_{9, 7}, F_{9, 8}, F_{9, 9}, F_{9, 10}, $$
and define a characteristic function $\lambda_2^{<12>}$ on the  ordered facets of $P^{12}$ by the  matrix
\begin{eqnarray*}
\left(
\begin{array}{cccccccccccccc}
1 &&&1&&&&&&&&&&\\
&1 & & -1&& &&&& &&&&   \\
 && 1&  1&&&&&& &&&&\\
 &&&-1& 1 &&&&&&&&&1\\
 &&&1&  &1 & &&&&&&& -1\\
 &&&&&&1&&&&&&&1\\
 &&&&&&&1&&&&&& -1\\
 &&&&&&&&1 &&&&& 1\\
 &&&&&&&&&1&&&&-1\\
 &&&&&&&&&&1&&&1\\
 &&&&&&&&&&&1&&-1\\
 &&&&&&&&&&&&1&1\\
\end{array}
\right),
\end{eqnarray*}
from which one obtains a special unitary manifold $M(P^{12}, \lambda_2^{<12>})$ with its  cohomology
$$H^*(M(P^{12}, \lambda_2^{<12>}))={\Bbb Z}[x,y]/<x^4, y^8(x-y)^2> \text{\rm with }\deg x=\deg y=2$$ and with its total Chern class
$$c(M(P^{12}, \lambda_2^{<12>}))=(1-x^2)^2(1-y^2)^4[1-(x-y)^2].$$
Then one has that the 6-th Chern class $c_6=-10y^6+12xy^5-26x^2y^4+16x^3y^3$,  so the Chern number $\langle c_6^2, [M(P^{12}, \lambda_2^{<12>})]\rangle=64\not=0$. Thus  $M(P^{12}, \lambda_2^{<12>})$
is not bordant to zero in $\Omega_*^U$.
\end{exam}

\subsection{Proof of Theorem~\ref{main}} Obviously,  the mod 2 reductions of the characteristic matrices $\Lambda_2^{(2, 0, ..., 0)}$ and
$\Lambda_2^{(4,2, 0, ..., 0)}$ of $M(P^{4l+5}, \lambda_2^{(2, 0, ..., 0)})$ and $M(P^{8l+11}, \lambda_2^{(4,2, 0, ..., 0)})$ are
\begin{eqnarray*}
\left(
\begin{array}{cccc}
I_2  &\textbf{1}_2 &    &\\
&J_1  &      I_{4l+3} & \textbf{1}_{4l+3}  \\
\end{array}
\right)
\end{eqnarray*}
and
\begin{eqnarray*}
\left(
\begin{array}{ccccccc}
I_4  &\textbf{1}_4 &   & &&\\
& &     I_2 &\textbf{1}_2&&&\\
&J_{1}   & &J_2&  &I_{8l+5} & \textbf{1}_{8l+5}  \\
\end{array}
\right)
\end{eqnarray*}
respectively. Thus, by Proposition~\ref{mp}, one has that the fixed point sets of the conjugation involutions on $M(P^{4l+5}, \lambda_2^{(2, 0, ..., 0)})$ and $M(P^{8l+11}, \lambda_2^{(4,2, 0, ..., 0)})$ are homeomorphic to the Stong manifolds
${\Bbb R}P(2, \underbrace{0, ..., 0}_{4l+3})$ and ${\Bbb R}P(4, 2, \underbrace{0, ..., 0}_{8l+4})$, respectively. Thus,
the subring of $\Omega_*^U$ generated by the unitary bordism classes $\beta_{8l+10}, \beta_{16l+22}$ of $M(P^{4l+5}, \lambda_2^{(2, 0, ..., 0)})$ and $M(P^{8l+11}, \lambda_2^{(4,2, 0, ..., 0)})$ is mapped onto the subring ${\Bbb Z}_2[\alpha_{4l+5}, \alpha_{8l+11}|l\geq 0]$ of $\mathfrak{N}_*$ in Lemma~\ref{subring} via $H_*: \Omega^U_{*}\longrightarrow \mathfrak{N}_*$. This means that any non-trivial polynomial in $\beta_{8l+10}$ and $\beta_{16l+22}$
is nonzero in $\Omega^U_{*}$ since its image under $H_*$ is  nonzero by Lemma~\ref{subring}, so we obtain the examples of non-bounding specially omnioriented quasitoric $(2n)$-manifolds with $n\not=1,2,3,4, 6, 7, 8, 12$.

For $n=7, 8, 12$, Examples~\ref{7}--\ref{12} directly provide  three non-bounding specially omnioriented quasitoric manifolds.
This completes the proof of Theorem~\ref{main}.
\hfill$\Box$

\begin{rem}

A counterexample in the case $n=6$ was recently discovered from a joint work~\cite{lp} of the first author with Taras Panov,  concerning the toric generators in the unitary and special unitary bordism rings.

\end{rem}

{\bf Acknowledgements}. The authors  express their gratitude to the referees,
who read this paper very carefully and gave many valuable suggestions and comments.

\end{document}